\documentclass[leqno,11pt]{amsart}
\usepackage{}
\usepackage{amssymb, amsmath}
\usepackage{amsthm, amsfonts,mathrsfs}
\usepackage{epsfig,subfigure}
\usepackage{ifthen}
\usepackage{dsfont}
\def\inte#1{
\displaystyle\mathop{#1\kern0pt}^\circ }

\def\virgp{\raise 2pt\hbox{,}}
\def\cdotpv{\raise 2pt\hbox{;}}

\def\C{\mathop{\mathbb C\kern 0pt}\nolimits}
\def\DD{\mathop{\mathbb D\kern 0pt}\nolimits}
\def\EE{\mathop{{\mathbb E \kern 0pt}}\nolimits}
\def\K{\mathop{\mathbb K\kern 0pt}\nolimits}
\def\N{\mathop{\mathbb N\kern 0pt}\nolimits}
\def\Q{\mathop{\mathbb Q\kern 0pt}\nolimits}
\def\R{\mathop{\mathbb R\kern 0pt}\nolimits}
\def\SS{\mathop{\mathbb S\kern 0pt}\nolimits}
\def\ZZ{\mathop{\mathbb Z\kern 0pt}\nolimits}
\def\TT{\mathop{\mathbb T\kern 0pt}\nolimits}
\def\P{\mathop{\mathbb P\kern 0pt}\nolimits}
\def \U{\textbf{u}}
\def \B{\textbf{b}}
\newcommand{\ds}{\displaystyle}

\newcommand{\bR}{\mathbb{R}}

\newcommand{\beq}{\begin{equation}}
\newcommand{\eeq}{\end{equation}}
\newcommand{\ben}{\begin{eqnarray}}
\newcommand{\een}{\end{eqnarray}}
\newcommand{\beno}{\begin{eqnarray*}}
\newcommand{\eeno}{\end{eqnarray*}}

\newtheorem{thm}{Theorem}[section]
\newtheorem{lem}{Lemma}[section]
\newtheorem{rmk}{Remark}[section]

\renewcommand{\theequation}{\thesection.\arabic{equation}}

\begin{document}

\title[]
{Liouville type Theorems for the stationary compressible 3D MHD
equations}

\author[J.-M. Kim]{Jae-Myoung Kim}

\thanks{Department of Mathematics Education, Andong National University, Andong 36729, Republic of Korea}
\thanks{E-mail address: jmkim02@anu.ac.kr}

\begin{abstract}
In this paper, we investigate the three dimensional stationary
compressible $3D$ MHD equations, and obtain Liouville type theorems
if a smooth solution $(\rho, \mathbf{u})$ satisfies some suitable
conditions. In particular, our results improve and generalize the
corresponding result.

\end{abstract}


\date{}

\maketitle



\vskip 0.2cm


\renewcommand{\theequation}{\thesection.\arabic{equation}}
\setcounter{equation}{0}
\section{Introduction}
The system of barotropic compressible magnetohydrodynamic(MHD)
equations in $\R^3$ can be read as follows:
\begin{equation}\label{MHD}
    \begin{cases}
    \operatorname{div}(\rho \mathbf{u})=0, \\
    - \nu\Delta \mathbf{u}-(\lambda + \nu)\nabla \operatorname{div} \mathbf{u} +\operatorname{div}(\rho \mathbf{u}\otimes \mathbf{u})-\operatorname{div}(\mathbf{b}\otimes \mathbf{b})+\nabla (P+\frac{|\mathbf{b}|^2}{2}) =0, \\
   - \Delta \mathbf{b}+\operatorname{div}(\rho \mathbf{u}\otimes \mathbf{b})-\operatorname{div}(\mathbf{b}\otimes \mathbf{u})=0,\\
   \operatorname{div}\mathbf{b}=0,
    \end{cases}
\end{equation}
where the vector $\mathbf{u}$ denotes the flow velocity field and
the scalar function $\rho$ represents the density of the fluid.
Without loss of generality, we consider the flows with $\gamma-$law
pressure:
$$P(\rho):=a\rho^\gamma, \quad a>0, \quad \gamma>1.$$
The shear viscosity $\nu$ and bulk viscosity $\lambda$ are both
constants and satisfy
$$\nu>0\ , \ \lambda+\frac{2}{3}\nu>0.$$ We consider the initial value problem of \eqref{MHD},
which requires initial conditions
\begin{equation}\label{ini}
\U(x,0)=\U_0(x)\quad \text{and} \quad \B(x,0)=\B_0(x), \qquad
x\in\R^3.
\end{equation}

Among them, we briefly recall the previous results concerned with
the compressible 3D MHD equations. The local strong solutions to the
compressible MHD with large initial data were obtained, by Vol'pert
and Khudiaev \cite{VK72} as the initial density is strictly positive
and also by Fan and Yu \cite{FW09} as the initial density may
contain vacuum, respectively. Hu and Wang \cite{HW08, HW08-1, HW09,
HW10} proved the global existence and uniqueness of classical
solutions to the Cauchy problem with smooth initial data which are
of small energy but possibly large oscillations. They obtain the
global existence of weak solutions for initial data which may be
discontinuous and contain vacuum states, and where the initial
temperature is allowed to be zero.  Later, Li, Xu and Zhang
\cite{LXZ13} shown if the initial energy is sufficiently small, the
global existence of a regular solution is shown. This solution is
classical with arbitrarily large oscillations, and contain vacuum
states.

On the contrary, to the best of one's knowledge, There is not much
result related to the equations \eqref{MHD}. Yang et. al.
\cite{YGD14} prove existence and uniqueness of strong solutions to
barotropic compressible magnetohydrodynamic (see also \cite{Yan16}
and \cite{Meng17}). Comparing to the steady incompressible MHD
equations, we refer to \cite{ZL04} and \cite{ZZ20}.  The long-time
behavior of a small global solution can be found in \cite{HW10},
\cite{WLY15}.

Regarding Liouville theorems for compressible $3D$ Navier-Stokes
equations, Chae \cite{Chae2012} showed if the smooth solution
$(\rho, \mathbf{u})$ satisfies
\[
     \begin{split}
\|\rho\|_{L^\infty(\mathbb{R}^3)}+\|\nabla
\mathbf{u}\|_{L^2(\mathbb{R}^3)}+\|\mathbf{u}\|_{L^\frac{3}{2}(\mathbb{R}^3)}<\infty.
     \end{split}
\]
then $u\equiv 0$ and $\rho=constant$. After that, through the
suitable decomposition for pressure (Lemma \ref{Lem3-1} below), Li
and Yu in \cite{L2014} shown
\[
     \begin{split}
\|\rho\|_{L^\infty(\mathbb{R}^3)}+\|\nabla \mathbf{u}\|_{L^2(\mathbb{R}^3)}+\|\mathbf{u}\|_{L^\frac{9}{2}(\mathbb{R}^3)}<\infty.
     \end{split}
\]
Later, in the Lorentz framework, Li and Niu \cite{LN21} proved if
$(\rho, \mathbf{u})$ is a smooth solution with $\rho\in L^\infty
(\mathbb{R}^3)$, $\nabla \mathbf{u}\in L^2(\mathbb{R}^3)$ and
$\mathbf{u}\in L^{p, q}(\mathbb{R}^3)$ for $3< p <\frac{9}{2}$,
$3\leq q\leq \infty$ or $p=q=3$. Then $\mathbf{u}\equiv 0$ and
$\rho=constant$ on $\mathbb{R}^{3}$. It is a relaxation result to
the corresponding result \cite{L2014}. In mathematics analysis,
Liouville theorems is classical problem. However, it has not been
resolved so far in a fluid equations.

 In present paper, we study
Liouviile Theorem for the compressible 3D MHD equations in a whose
space in viewpoint of Lorentz space or Morrey space based on the
Cacciopoly-type estimates. In our analysis, in particular, we use
pressure decomposition \eqref{pressure-decom} based on Lemma
\ref{Lem3-1} below, which is a essential tool to deal with the
pressure.

Before we look at the main results, we give some definitions for
functional spaces. Given $1\leq p< \infty , 1\leq q\leq \infty$, we
say that a measurable function $f \in L^{p, q}(\mathbb{R}^3)$ if
$\|f\|_{L^{p, q}(\mathbb{R}^3)}< \infty,$ where
\begin{equation*}\label{2-1a}
\|f\|_{L^{p, q}(\mathbb{R}^3)}:=
    \begin{cases}
        \left(\int_0^{\infty}t^{q-1}|\{x\in \mathbb{R}^3: |f(x)|> t\}|^{\frac{q}{p}}\, dt\right)^{\frac{1}{q}}, \quad \mathrm{if}\, \quad q<+\infty, \\
        \sup\limits_{t>0} t|\{x\in \mathbb{R}^3: |f(x)|> t\}|^{\frac{1}{p}}, \quad \mathrm{if}\,  \quad q=+\infty.
    \end{cases}
\end{equation*}
The space satisfies the continuous embedding (see e.g. \cite{L1950})
\begin{equation}\label{relarion-ll}
  L^{p}(\mathbb{R}^3)=L^{p, p}(\mathbb{R}^3)\hookrightarrow L^{p, q}(\mathbb{R}^3)\hookrightarrow L^{p, \infty}(\mathbb{R}^3) \ , \  p\leq q <\infty.
\end{equation}

In this direction, parallel to the result of Li and Niu, first
result is stated as
\begin{thm}\label{thm-1}
Suppose that $(\rho, \mathbf{u}, \mathbf{b})$ is a smooth solution
to \eqref{MHD} with $\rho\in L^\infty (\mathbb{R}^3)$, $\nabla
\mathbf{u}, \nabla \mathbf{b}\in L^2(\mathbb{R}^3)$ and
$\mathbf{u},\ \mathbf{b}\in L^{p, q}(\mathbb{R}^3)$ for $3< p
<\frac{9}{2}$, $3\leq q\leq \infty$ or $p=q=3$. Then
$\mathbf{u}\equiv 0$ and $\rho=constant$ on $\mathbb{R}^{3}$.
\end{thm}

In light of \cite[Theorem 1.2]{LLN20}, for $p\geq \frac{9}{2}$,  we
are easily checked to obtain a following results and omit a proof.
\begin{thm}\label{thm-2}
Let Let $\mathcal{C}(R/ 2, R) = \{ x \in \R^3: R/2 < \vert x \vert <
R \}$ and
$$M_{p, q}(R):=
R^{\frac{2}{3}-\frac{3}{p}}(\|\mathbf{u}\|_{L^{p,
q}(\mathcal{C}(R/2,R))}+\|\mathbf{b}\|_{L^{p,
q}(\mathcal{C}(R/2,R))}).$$ Suppose that $(\rho, \mathbf{u},
\mathbf{b})$ is a smooth solution to \eqref{MHD} with $\rho\in
L^\infty (\mathbb{R}^3)$ and $\nabla \mathbf{u},\ \nabla
\mathbf{b}\in L^2(\mathbb{R}^3).$  For $p\geq \frac{9}{2}$, $3\leq
q\leq \infty$, assume that
\begin{equation*}
     \begin{split}
\liminf_{R \rightarrow \infty} M_{\,p, \,q}(R)<\infty,
     \end{split}
\end{equation*}
then
\begin{equation}\label{1-1a}
     \begin{split}
D(\mathbf{u},\mathbf{b}):=\int_{\mathbb{R}^3}|\nabla \mathbf{u}|^2
+|\nabla \mathbf{b}|^2 \, dx \leq \mathcal{C}\liminf_{R \rightarrow
\infty}M^3_{\,p, \,q}(R).
     \end{split}
\end{equation}
If moreover assume
\begin{equation}\label{1-1}
     \begin{split}
\liminf_{R \rightarrow \infty}M^3_{\,p, \,q}(R)\leq \delta
D(\mathbf{u})
\end{split}
     \end{equation}
for some $0<\delta<\frac{1}{\mathcal{C}},$ then $\mathbf{u}\equiv 0$
and $\rho=constant$ on $\mathbb{R}^{3}$.
\end{thm}

Next, we also introduce the definition of the (homogeneous) Morrey
and local Morrey spaces (see e.g. \cite{PGLR1} for more details).
Let $1<p<r<+\infty$, the homogeneous Morrey space
$\dot{M}^{p,r}(\R^3)$ is the set of functions $f \in
L^{p}_{loc}(\R^3)$ such that
\begin{equation*}\label{Def-Morrey}
\Vert f \Vert_{\dot{M}^{p,r}}= \sup_{R>0,\,\, x_0 \in \R^3}
R^{\frac{3}{r}} \left(  \frac{1}{R^{3}} \int_{B(x_0, R)} \vert f (x)
\vert^p dx \right)^{\frac{1}{p}}< +\infty,
\end{equation*} where $B(x_0,R)$ denotes the ball centered at $x_0$ and with radio $R$.
It is well known  $L^{r}(\R^3)\subset L^{r, q}(\R^3)  \subset \dot{M}^{p,r}(\R^3)$, where, for $r \leq q \leq +\infty$.  \\

For $\gamma\geq 0$ and $1<p<+\infty$, we define the \emph{local
Morrey space} $M^{p}_{\gamma}(\R^3)$ as the Banach space of
functions $f \in L^{p}_{loc}(\R^3)$ such that
\begin{equation*}\label{def-local-morrey}
\Vert f \Vert_{M^{p}_{\gamma}} = \sup_{R \geq 1} \left(
\frac{1}{R^{\gamma}} \int_{B(0,R)} \vert f (x) \vert^p
dx\right)^{1/p} <+\infty.
\end{equation*}
Moreover,
\[
\gamma_1 \leq \gamma_2 \Rightarrow M^{p}_{\gamma_1}(\R^3) \subset
M^{p}_{\gamma_2} (\R^3),
\]
and also
\[
1 < p < r <+\infty \Rightarrow \ds{\dot{M}^{p,r}}(\R^3) = M^{p}_{3(
1-p/r)}(\R^3)  \subset M^{p}_{\gamma}(\R^3),
\]
where $\gamma>0$ is constant with $3(1-p/r)<\gamma$.

Finally,  we define the space $M^{p}_{\gamma,0}(\R^3)$ as the set of
functions $f \in M^{p}_{\gamma}(\R^3)$ such that
 \begin{equation*}\label{cond-dec1}
 \lim_{R \to +\infty}  \left( \frac{1}{R^\gamma} \int_{\mathcal{C}(R/2, R)} \vert f (x) \vert^p dx \right)^{1/p} =0.
 \end{equation*}
Relationship between these spaces \cite{Jarrin20-1} is simply known
that for $3<r<9/2$ and for $3(1-3/r)<\delta<1$,
$$L^r(\R^3)\subset L^{r,\infty}(\R^3)\subset \dot{M}^{3,r}(\R^3)\subset
M^{3}_{\delta}(\R^3)\subset M^{3}_{1,0}(\R^3),$$ and for $r=9/2$ and
$9/2<q<+\infty$, $L^{9/2}(\R^3)\subset L^{9/2, q}(\R^3)\subset
M^{3}_{1,0}(\R^3)$.

In the framework of Morrey space, third result is followed as

\begin{thm}\label{Theo:Morrey-sous-critique}  Let $\U , \B\in L^{2}_{loc}(\R^3)$ be a weak solution of \eqref{MHD}.
If $\U, \B \in \dot{M}^{p,r}(\R^3)$ with $3\leq p < r <
\frac{9}{2}$, then $\U \equiv 0\equiv \B$ in $\R^3$.
\end{thm}

In the framework of local Morrey space, third result is followed as

\begin{thm}\label{Th1} Let $(\rho, \U, \mathbf{b})$
be a smooth solution of \eqref{MHD}.
\begin{enumerate}
    \item[$1)$] Let $\gamma=1$.  If $\U \in M^{3}_{1,0}(\R^3)$ and $\nabla \mathbf{b} \in  M^{3}_{1}(\R^3)$ then we have $\U=0$ and  $\nabla \mathbf{b} =0$.
    \item[$2)$] Let $1<\gamma<3/2$.  If $\U \in M^{3}_{\gamma,0}(\R^3)$ and $\nabla \mathbf{b} \in  M^{3}_{\gamma}(\R^3)$, and moreover if the velocity $\U$ verifies:
    \begin{equation*}\label{cond-dec}
 \lim_{R \to +\infty}  R^{\gamma-1}\left( \frac{1}{R^{\gamma}} \int_{\mathcal{C}(R/2, R)} \vert \U(x) \vert^3 dx \right)^{1/3} =0,
\end{equation*}
    then we have  $\U=0$ and  $\mathbf{b}=0$.
\end{enumerate}
\end{thm}

When $\rho=constant$ , the system of \eqref{MHD} reduce as follows:
\begin{equation}\label{MHD-incompressible}
    \begin{cases}
    - \nu\Delta \mathbf{u}+\operatorname{div}(\mathbf{u}\otimes \mathbf{u})+\nabla \pi =0, \\
   - \Delta \mathbf{b}+\operatorname{div}( \mathbf{u}\otimes \mathbf{b})-\operatorname{div}(\mathbf{b}\otimes \mathbf{u})=0,\\
   \operatorname{div}\mathbf{u}=0=\operatorname{div}\mathbf{b},
    \end{cases}
\end{equation}
where $\pi$ is a total pressure.

For the incompressible 3D MHD equations \eqref{MHD-incompressible},
 Schulz \cite{Schulz19} proved a Liouville theorem  without the finite Dirichlet integral, which is if the smooth
solution
\[
(u, b) \in (L^p\cap BMO^{-1})(\R^3),\quad p \in (2, 6],
\] then
$u=b =0$. Recently, Yuan and Xiao \cite{Yuan20} shown if
\[
(u, b) \in L^p(\R^3),\quad 2\leq p \leq 9/2,
\]
then $u =b =0$. For a interesting results, we also refer to
\cite{LZ18} in spite of results in \cite{KNSS09} and \cite{LLN20}
inspired by the work in \cite{CJL-R} in the Lorentz framework.

In this directions, two results are stated in the framework of
Lebesgue and Lorentz spaces.

\begin{thm}\label{Theo:Lorentz} Let $\U, \B \in L^{2}_{loc}(\R^3)$ be a weak solution of \eqref{MHD-incompressible}.
\begin{enumerate}
\item[1)] If $\U, \B \in L^{9/2, q}(\R^3)$, with $9/2 \leq q <+\infty$, then we have $\U=0=\B$.
\item[2)] If $\U \in L^{r,q}$ with $9/2<r\leq q <+\infty$ and $\B \in L^{9/2, q}(\R^3)$, with $9/2 \leq q <+\infty$, and in addition, if
    \begin{equation*}
    \sup_{R>1} R^{2-\frac{9}{r}} \Vert \U \Vert_{L^{r,q}(\mathcal{C}(R/4, 2R))} <+\infty,
    \end{equation*}
     then we have $\U=0=\B$.
\end{enumerate}
\end{thm}

Investigating solutions of the equations \eqref{MHD} that may decay
to zero in different rates as $|x| \rightarrow \infty$ with respect
to different directions, we introduce a new functional space
\cite{Phan20}. For two given numbers $q, r \in (1, \infty)$ and $f :
\mathbb{R}^3 \rightarrow \mathbb{R}$ is a measurable function,
\emph{the mixed-norm Lebesgue space} $L_{q,r}(\R^3)$  is the space
that is equipped with the following norm
\[
\|f\|_{L_{q,r}(\R^3)} = \left [\int_{\R} \left (\int_{\R^2} |f(x_1,
x_2, x_3)|^q dx_1 dx_2 \right)^{r/q} dx_3 \right]^{1/r}.
\]

\begin{thm} \label{mixed-norm} Let $q, r \in [3, \infty)$ be two numbers satisfying
\begin{equation*} \label{q-r-relation}
\frac{2}{q} + \frac{1}{r} \geq \frac{2}{3}.
\end{equation*}
and  $(\U, \B) \in H^1_{\textup{loc}}(\R^3)$ be a weak solution of
\eqref{MHD-incompressible} and assume that $(\U, \B) \in
L_{q,r}(\R^3)$. Then $\U \equiv 0\equiv \B$ in $\R^3$.
\end{thm}

\begin{rmk}
The result given in \ref{Theo:Lorentz} is of particular interest
since this result can be regarded as a improvement  of the results
given in \cite{KzonoAL} and \cite{SerWang}.
\end{rmk}

\begin{rmk}
In case of the compressible fluids, we don't know if Theorem
  \ref{Theo:Lorentz} or \ref{mixed-norm} are true yet. For a proof of Theorem
  \ref{Theo:Lorentz}, because it is necessary to get the estimate of $\|\nabla
  u\|_{L^p}$. More speaking, to control a estimate of $\|\nabla
  u\|_{L^p}$, it needs a higher regularity of $\rho$, in succession, which
  yields that it needs a higher regularity of $u$. On the other
  hands, for a proof of Theorem, because it is necessary to use Lemma \ref{mixed-norm} based on the weight function.
\end{rmk}

Lastly, we make a few comments.

\begin{rmk}
The present paper focuses on Liouville theorem for the compressible
MHD equation, not Hall-MHD equation.  We think it is secondary to
consider the Hall term in the system \eqref{MHD}. Considering the
Hall term, $\mbox{curl}\times (\text{curl}\mathbf{b}\times
\mathbf{b})$ in \eqref{MHD}, namely, (compressible) Hall-MHD
equations, we can see that additional conditions are required for
the magnetic filed $b$ in light of the references \cite{Zeng18},
\cite{Yuan20} and \cite{LN20}. Proofs of this part is not difficult,
and thus leave it to the readers. For this equation, we refer to
\cite{CW15} and \cite{ZZ20} for interested readers.
\end{rmk}

\begin{rmk}
For a various coupled equation with (compressible) Naiver-Stokes
equations, the above Theorems are applicable through Cacciopoly type
estimate.
\end{rmk}


\renewcommand{\theequation}{\thesection.\arabic{equation}}
\setcounter{equation}{0} 

\section{Proof of Theorem \ref{thm-1}}

The following inequalities in Lorentz spaces are useful.

\begin{lem}[\cite{O1963}]\label{Lem2-1}
\label{lem} Let $f\in L^{p_1,q_1}(\mathbb{R}^3)$ and $g\in
L^{p_2,q_2}(\mathbb{R}^3)$ with $1\leq p_1, p_2\leq\infty$, $1\leq
q_1, q_2\leq\infty.$ Then $fg\in L^{p,q}(\mathbb{R}^3)$ with
$\frac{1}{p}=\frac{1}{p_1}+\frac{1}{p_2}$,
$\frac{1}{q}\leq\frac{1}{q_1}+\frac{1}{q_2}$ and
\begin{equation*}\label{2-1}
     \begin{split}
   \|fg\|_{L^{p,q}(\mathbb{R}^3)}\leq C\|f\|_{L^{p_1,q_1}(\mathbb{R}^3)}\|g\|_{L^{p_2,q_2}(\mathbb{R}^3)}
     \end{split}
\end{equation*}
for a constant $C>0.$
\end{lem}

\begin{lem}[\cite{B1976}]\label{Lem2-2}
Let $\Omega$ be a bounded domain in $\mathbb{R}^n$, $1<p<\infty$,
$1<q\leq\infty$ and $f\in L^{p, q}(\Omega)$. Then
\begin{equation*}\label{2-2}
     \begin{split}
   \|\nabla^2(-\Delta)^{-1}f\|_{L^{p,q}(\Omega)}\leq C\|f\|_{L^{p,q}(\Omega)},
     \end{split}
\end{equation*}
where the constant $C>0$ is independent of $\Omega$.
\end{lem}

To obtain the estimate for the pressure term, we recall the
following lemma from \cite{L2014}.
\begin{lem}\label{Lem3-1}
\label{lem} Let $P\in L^{\infty}(\mathbb{R}^3),$ $p_1\in
L^{r_1}(\mathbb{R}^3),$ $p_2\in L^{r_2}(\mathbb{R}^3)$ with $1\leq
r_1, r_2<\infty.$ Suppose that $P-p_1-p_2$ is weakly harmonic, that
is
$$\Delta (P-p_1-p_2)=0$$
in the sense of distribution, then there exists a constant $c$ such
that
\begin{equation*}
     \begin{split}
      P-p_1-p_2 = c \quad\, \mbox{a.e.}\quad\, x\in \mathbb{R}^3
      \end{split}
\end{equation*}
If furthermore $P(x)\geq 0$ a.e., then we also have $c\geq 0$.
\end{lem}

Taking the div operation on both sides of $\eqref{MHD}_2$, we have
$
   \Delta (P-p_1-p_2)=0
$, where
\begin{equation}\label{pressure-decom}
   p_1:=(-\Delta)^{-1}\partial_i\partial_j (\rho u_iu_j)\quad
   \mbox{and}\quad p_2:=(\lambda+2\nu)\operatorname{div}\mathbf{u}.
\end{equation}
Using the assumption $\nabla \mathbf{u}\in L^2{(\mathbb{R}^3)}$ and
the Sobolev embedding $\dot{H}^1(\mathbb{R}^3)\hookrightarrow
L^6(\mathbb{R}^3)$, it follows
$$p_1\in L^3(\mathbb{R}^3)\ , \ p_2\in L^2(\mathbb{R}^3).$$
Due to Lemma \ref{Lem3-1}, there exists a constant $c\geq 0$ such
that
$$ a\rho^\gamma=P=c+p_1+p_2.$$
Considering the function
$P_1:=\rho^{\gamma-1}-(\frac{c}{a})^{\frac{\gamma-1}{\gamma}}=(\frac{c+p_1+p_2}{a})^{\frac{\gamma-1}{\gamma}}-(\frac{c}{a})^{\frac{\gamma-1}{\gamma}}$,
we have
\begin{equation}\label{pressure-relation}
\nabla P=\nabla(a\rho
^\gamma)=\frac{a\gamma}{\gamma-1}\rho\nabla(\rho^{\gamma-1})=\frac{a\gamma}{\gamma-1}\rho\nabla
P_1
\end{equation}
 and
\begin{equation}\label{3-6a}
     \begin{split}
  |P_1\rho|\leq (|p_1|+|p_2|).
     \end{split}
\end{equation}

\begin{proof} [Proof of Theorem \ref{thm-1}] Following the same arguments in  \cite{LN20} or \cite{L2014}, we are aimed to show a suitable Caccioppoli-type inequality.
For the convenience of our readers, we give a sketch of the proof.
Let $\varphi \in C_0^\infty (\mathbb{R}^3)$ be a radial cut-off
function satisfying
\begin{equation*}\label{3-1}
\varphi(|x|)=
    \begin{cases}
        1, \quad \mathrm{if}\,  |x|< \frac{1}{2}, \\
        0, \quad \mathrm{if}\,  |x|> 1,
    \end{cases}
\end{equation*}
and $0 \leq \varphi (|x|) \leq 1 $ for $\frac{1}{2}\leq |x|\leq 1.$
For each given $R>0$, we define
$\varphi_R(x):=\varphi(\frac{|x|}{R})$ satisfying
$$\|\nabla^k \varphi_R\|_{L^\infty} \leqslant CR^{-k}, \quad k=0, 1, 2$$
for some constant $C>0$ independent of $x\in \mathbb{R}^3$.

Taking the inner product of  $\eqref{MHD}_2$ with
$\mathbf{u}\varphi_R$ and $\eqref{MHD}_3$ with $\mathbf{b}\varphi_R$
and after integrating by parts, it follows that
\begin{equation}\label{3-3-1}
     \begin{split}
   &\int_{B_{\frac{R}{2}}}\Big(\nu|\nabla \mathbf{u}|^2+|\nabla \mathbf{b}|^2\Big) \varphi_R\, dx + (\lambda+\nu)\int_{\frac{R}{2}}|\operatorname{div} \mathbf{u}|^2\, dx\\
   &=-\nu\int_{\mathcal{C}(R/2,R)}(|\mathbf{u}|^2+|\mathbf{b}|^2) \Delta \varphi_R \, dx
   -2(\lambda+\nu)\int_{\mathcal{C}(R/2,R)} \operatorname{div} \mathbf{u} \mathbf{u}\cdot \nabla \varphi_R \, dx\\
   &\quad-\int_{\mathcal{C}(R/2,R)}|\mathbf{u}|^2 \nabla \varphi_R\cdot\rho \mathbf{u} \, dx +\int_{\mathcal{C}(R/2,R)}|\mathbf{b}|^2 \nabla \varphi_R\cdot \mathbf{u} \, dx
   - \int_{\mathcal{C}(R/2,R)}\nabla P\cdot \mathbf{u}\varphi_R\,
dx:=\sum_{i=1}^6
   \mathcal{I}_i,
     \end{split}
\end{equation}
where we use $\operatorname{div}(\rho \mathbf{u})=0$ and by the
integration by parts and chain rule,
\begin{equation*}
     \begin{split}
  &-2\nu\int_{\mathbb{R}^3} \nabla \mathbf{u} : \mathbf{u}\otimes \nabla \varphi_R \, dx=\nu\int_{\mathbb{R}^n}|\mathbf{u}|^2 \Delta (\varphi_R) \,
   dx.
     \end{split}
\end{equation*}
Now, we estimate $ \mathcal{I}_i$ term by term. we assume $ p > 3$,
$3\leq q\leq \infty$ or $p=q=3$. For $ \mathcal{I}_1$, since
\begin{equation*}\label{3-4}
     \begin{split}
   \int_{\mathcal{C}(R/2,R)}|\mathbf{u}|^2 |\Delta \varphi_R| \, dx
   \lesssim  R^{1-\frac{6}{p}}\|\mathbf{u}\|^2_{L^{p,
q}(\mathcal{C}(R/2,R))},
     \end{split}
\end{equation*}
we have
\begin{equation*}\label{3-4}
     \begin{split}
   | \mathcal{I}_1|\lesssim  R^{1-\frac{6}{p}}\Big(\|\mathbf{u}\|^2_{L^{p,
q}(\mathcal{C}(R/2,R))}+\|\mathbf{b}\|^2_{L^{p,
q}(\mathcal{C}(R/2,R))}\Big).
     \end{split}
\end{equation*}
For $ \mathcal{I}_2$, by H\"{o}lder and  Young inequalities, it
yields
\begin{equation*}\label{3-5}
     \begin{split}
   | \mathcal{I}_2|&\leq C(\lambda+\nu) \int_{\mathcal{C}(R/2,R)}| \operatorname{div} \mathbf{u}|| \mathbf{u}||\nabla \varphi_R| \, dx\\
   &\leq \frac{(\lambda+\nu)}{2} \| \operatorname{div} \mathbf{u}\|_{L^2(\mathbb{R}^3)}^2+CR^{1-\frac{6}{p}}\|\mathbf{u}\|^2_{L^{p, q}(\mathcal{C}(R/2,R))}.\\
     \end{split}
\end{equation*}
For $ \mathcal{I}_3$, we have 
\begin{equation*}\label{3-6}
     \begin{split}
   | \mathcal{I}_3|&\lesssim R^{-1}\|\rho\|_{L^\infty}\||\mathbf{u}|^3\|_{L^{\frac{p}{3}, \frac{q}{3}}(\mathcal{C}(R/2,R))}
   \|1\|_{L^{\frac{p}{p-3}, \frac{q}{q-3}}(\mathcal{C}(R/2,R))}\leq C R^{2-\frac{9}{p}}\|\mathbf{u}\|_{L^{p,
q}(\mathcal{C}(R/2,R))}^3.
     \end{split}
\end{equation*}
In a same manner as $ \mathcal{I}_3$, by the integration by parts, $
\mathcal{I}_4$ and $ \mathcal{I}_5$ are rewritten by
\begin{equation*}\label{3-6}
     \begin{split}
   | \mathcal{I}_4|+| \mathcal{I}_5|\lesssim R^{2-\frac{9}{p}}\Big(\|\mathbf{u}\|_{L^{p, q}(\mathcal{C}(R/2,R))}^3+\|\mathbf{b}\|_{L^{p, q}(\mathcal{C}(R/2,R))}^3\Big).
     \end{split}
\end{equation*}

Integrating by parts with \eqref{pressure-decom} and
\eqref{pressure-relation}, $\mathcal{I}_6$ becomes
\begin{equation*}
     \begin{split}
   \mathcal{I}_6&=-\int_{\mathbb{R}^3}\frac{a\gamma}{\gamma-1}\rho\nabla P_1\cdot \mathbf{u}\varphi_R \, dx\\
   &=\frac{a\gamma}{\gamma-1}\int_{\mathbb{R}^3}P_1\operatorname{div}(\rho \mathbf{u})\varphi_R\, dx+
   \frac{a\gamma}{\gamma-1}\int_{\mathbb{R}^3}P_1 \rho\mathbf{u}\cdot\nabla(\varphi_R)\, dx\\
   &=\frac{2a\gamma}{\gamma-1}\int_{\mathbb{R}^3}P_1 \rho\mathbf{u}\cdot(\varphi_R\nabla \varphi_R)\, dx,
     \end{split}
\end{equation*}
and thus we have
\begin{equation*}\label{compressible_pressure}
     \begin{split}
  |\mathcal{I}_6|\leq R^{-1}\int_{\mathcal{C}(R/2,R)}(|p_1|+|p_2|)|\mathbf{u}|\, dx.
  \end{split}
\end{equation*}
By Lemmas \ref{Lem2-1} and \ref{Lem2-2}, we have
\begin{equation*}\label{3-6b}
     \begin{split}
  R^{-1}\int_{\mathcal{C}(R/2,R)}|p_1||\mathbf{u}|\, dx&\leq CR^{-1}\|p_1\|_{L^{\frac{p}{2}, \frac{q}{2}}(\mathcal{C}(R/2,R))}\|\mathbf{u}\|_{L^{p, q}(\mathcal{C}(R/2,R))}\|1\|_{L^{\frac{p}{p-3}, \frac{q}{q-3}}(\mathcal{C}(R/2,R))}\\
  &\leq CR^{2-\frac{9}{p}}\|\mathbf{u}\|^3_{L^{p, q}(\mathcal{C}(R/2,R))}
  \end{split}
\end{equation*}
and
\begin{equation*}\label{3-6c}
     \begin{split}
  R^{-1}\int_{\mathcal{C}(R/2,R)}|p_2||\mathbf{u}|\, dx&\leq R^{-1}\|p_2\|_{L^{\frac{p}{p-1}, \frac{q}{q-1}}(\mathcal{C}(R/2,R))}\|\mathbf{u}\|_{L^{p, q}(\mathcal{C}(R/2,R))}\\
  &\leq CR^{\frac{1}{2}-\frac{3}{p}}\|\operatorname{div}\mathbf{u}\|_{{L^2}(\mathbb{R}^3)}\|\mathbf{u}\|_{L^{p, q}(\mathcal{C}(R/2,R))}\\
  &\leq \frac{\nu}{2}\|\nabla \mathbf{u}\|^2_{{L^2}(|x|\leq2R)}+CR^{1-\frac{6}{p}}\|\mathbf{u}\|^2_{L^{p, q}(\mathcal{C}(R/2,R))}.
  \end{split}
\end{equation*}
Therefore,
\begin{equation*}\label{3-7}
     \begin{split}
   |\mathcal{I}_6|&\lesssim R^{2-\frac{9}{p}}\|\mathbf{u}\|^3_{L^{p, q}(\mathcal{C}(R/2,R))}+CR^{1-\frac{6}{p}}\|\mathbf{u}\|_{L^{p, q}(\mathcal{C}(R/2,R))}.
     \end{split}
\end{equation*}

Substituting all estimates of $\mathcal{I}_1$--$\mathcal{I}_6$ into
\eqref{3-3-1} leads to
\begin{equation}\label{3-8-1}
     \begin{split}
     &\nu \int_{|x|\leq R}|\nabla \mathbf{u}|^2\, dx + \frac{(\lambda+\nu)}{2}\int_{|x|\leq R}|\operatorname{div} \mathbf{u}|^2 \, dx \\
     &\lesssim \left(R^{1-\frac{6}{p}}(\|\mathbf{u}\|^2_{L^{p,
     q}(\mathcal{C}(R/2,R))}+\|\mathbf{b}\|^2_{L^{p,
     q}(\mathcal{C}(R/2,R))})
     + R^{2-\frac{9}{p}}(\|\mathbf{u}\|^3_{L^{p, q}(\mathcal{C}(R/2,R))}+\|\mathbf{b}\|^3_{L^{p, q}(\mathcal{C}(R/2,R))})\right)
     \end{split}
\end{equation}
for all $R > 0$.

Passing $R\rightarrow +\infty$ in \eqref{3-8-1}, due to the
assumptions,  we have
\begin{equation*}
     \begin{split}
     \lim_{R\rightarrow +\infty}(\int_{B_{\frac{R}{2}}}|\nabla \mathbf{u}|^2 \, dx + \int_{\frac{R}{2}}|\operatorname{div} \mathbf{u}|^2 \, dx)
     =0,
     \end{split}
\end{equation*}
which implies that $\mathbf{u}\equiv 0\equiv\mathbf{b}$. On the
other hand, we know that $\nabla (a\rho^\gamma)=0$, which implies
that $\rho$=constant on $\mathbb{R}^3$ by means of $\eqref{MHD}_2$.
The proof of Theorem \ref{thm-1} is ended.
\end{proof}

%

\renewcommand{\theequation}{\thesection.\arabic{equation}}
\setcounter{equation}{0} 


\section{Morrey space-Proof of Theorem \ref{Theo:Morrey-sous-critique}}
From \eqref{3-8-1} with the relationship between spaces
\eqref{relarion-ll}, for all $R
> 0$, we know
\begin{equation}\label{3-8-2}
     \begin{split}
     &\nu \int_{|x|\leq R}|\nabla \mathbf{u}|^2\, dx + \frac{(\lambda+\nu)}{2}\int_{|x|\leq R}|\operatorname{div} \mathbf{u}|^2 \, dx \\
     &\lesssim \left(R^{1-\frac{6}{p}}(\|\mathbf{u}\|^2_{L^{p}(\mathcal{C}(R/2,R))}+\|\mathbf{b}\|^2_{L^{p}(\mathcal{C}(R/2,R))})
     + R^{2-\frac{9}{p}}(\|\mathbf{u}\|^3_{L^{p}(\mathcal{C}(R/2,R))}+\|\mathbf{b}\|^3_{L^{p}(\mathcal{C}(R/2,R))})\right).
     \end{split}
\end{equation} 
Using the definition of Morrey space, the inequality \eqref{3-8-2}
becomes
\begin{equation*}
     \begin{split}
   &\int_{B_{\frac{R}{2}}}\Big(|\nabla \mathbf{u}|^2+|\nabla \mathbf{b}|^2\Big) \varphi_R\, dx +\int_{\frac{R}{2}}|\operatorname{div} \mathbf{u}|^2\, dx\\
   &  \lesssim R^{2-\frac{9}{r}}(\|\mathbf{u}\|^2_{\dot{M}^{p,r}}+\|\mathbf{b}\|^2_{\dot{M}^{p,r}})
   \|\mathbf{u}\|_{\dot{M}^{p,r}}
   +R^{1-\frac{6}{r}}\Big(\|\mathbf{u}\|^2_{\dot{M}^{p,r}}+\|\mathbf{b}\|^2_{\dot{M}^{p,r}}\Big),
     \end{split}
\end{equation*}
which implies $\mathbf{u}\equiv0\equiv\mathbf{b}$ in $\R^3$ as
$R\rightarrow\infty$. The proof of Theorem
\ref{Theo:Morrey-sous-critique} is complete.

\section{Proof of Theorem \ref{Th1}}

From \eqref{3-8-1}, for all $R > 0$, we know
\begin{equation}\label{3-8-3}
     \begin{split}
     &\nu \int_{|x|\leq R}|\nabla \mathbf{u}|^2\, dx + \frac{(\lambda+\nu)}{2}\int_{|x|\leq R}|\operatorname{div} \mathbf{u}|^2 \, dx \\
     &\lesssim \left(R^{1-\frac{6}{p}}(\|\mathbf{u}\|^2_{L^{p}(\mathcal{C}(R/2,R))}+\|\mathbf{b}\|^2_{L^{p}(\mathcal{C}(R/2,R))})
     + R^{2-\frac{9}{p}}(\|\mathbf{u}\|^3_{L^{p}(\mathcal{C}(R/2,R))}+\|\mathbf{b}\|^3_{L^{p}(\mathcal{C}(R/2,R))})\right).
     \end{split}
\end{equation}
Note that by definition of the local Morrey spaces
$M^{3}_{\gamma,0}(\R^3)$, we have $\U , \mathbf{b}\in
L^{3}_{loc}(\R^3)$ and thus, letting $p=3$, the estimate
\eqref{3-8-3} becomes
\begin{equation*}\label{3-3}
     \begin{split}
   &\int_{B_{\frac{R}{2}}}\Big(|\nabla \mathbf{u}|^2+|\nabla \mathbf{b}|^2\Big) \varphi_R\, dx +\int_{\frac{R}{2}}|\operatorname{div} \mathbf{u}|^2\, dx\\
   &  \lesssim R^{-1}(\|\mathbf{u}\|^3_{L^{3}(\mathcal{C}(R/2,R))}+\|\mathbf{b}\|^3_{L^{3}(\mathcal{C}(R/2,R))})
   +CR^{-1}(\|\mathbf{u}\|^2_{L^{3}(\mathcal{C}(R/2,R))}+\|\mathbf{b}\|^2_{L^{3}(\mathcal{C}(R/2,R))}).
     \end{split}
\end{equation*}
And thus, we again write for $1\leq \gamma<3/2$
\begin{eqnarray}
& &\int_{B_{\frac{R}{2}}}\Big(|\nabla \mathbf{u}|^2+|\nabla \mathbf{b}|^2\Big) \varphi_R\, dx +\int_{\frac{R}{2}}|\operatorname{div} \mathbf{u}|^2\, dx\nonumber\\
&\lesssim&
R^{-1}(\|\mathbf{u}\|^2_{L^{2}(\mathcal{C}(R/2,R))}+\|\mathbf{b}\|^2_{L^{2}(\mathcal{C}(R/2,R))})
+\left(\frac{1}{R^{\gamma}} \int_{\mathcal{C}(R/2, R)} \vert \U
\vert^3 dx\right)^{2/3} R^{\gamma-1} \left( \frac{1}{R^{\gamma}}
\int_{\mathcal{C}(R/2, R)} \vert \U \vert^3 dx \right)^{1/3}\nonumber\\
&+&  \left(\frac{1}{R^{\gamma}} \int_{\mathcal{C}(R/2, R)} \vert \B
\vert^3 dx\right)^{2/3} R^{\gamma-1} \left( \frac{1}{R^{\gamma}}
\int_{\mathcal{C}(R/2, R)} \vert \B \vert^3 dx \right)^{1/3}\nonumber\\
&\lesssim&
R^{-1}\Big(\|\mathbf{u}\|^2_{L^{2}(\mathcal{C}(R/2,R))}+\|\mathbf{b}\|^2_{L^{2}(\mathcal{C}(R/2,R))}\Big)\nonumber\\
&+&   R^{\gamma-1}\left[\Vert \U \Vert^{2}_{M^{3}_{\gamma}} \left(
\frac{1}{R^{\gamma}} \int_{\mathcal{C}(R/2, R)} \vert \U \vert^3 dx
\right)^{1/3}+ \Vert \mathbf{b} \Vert^{2}_{M^{3}_{\gamma}}  \left(
\frac{1}{R^{\gamma}} \int_{\mathcal{C}(R/2, R)} \vert \B \vert^3 dx
\right)^{1/3}\right]\label{Lm-20}.
\end{eqnarray}
We shall study each term in the right-hand side. For the first and
second term in \eqref{Lm-20}, we have
\begin{equation}\label{lim-1}
\lim_{R \to +\infty} \frac{c}{R^2}\int_{\mathcal{C}(R/2,R)}\vert \U
\vert^2 dx=0,\quad \mbox{and}\quad \lim_{R \to +\infty}
\frac{c}{R}\Big(\int_{\mathcal{C}(R/2,R)}\vert \U \vert^3
dx\Big)^{2/3}=0.
\end{equation}
Indeed,
$$ \frac{c}{R^2} \int_{\mathcal{C}(R/ 2, R)}\vert \U \vert^{2}dx \leq c\, R^{-1} \left( \int_{\mathcal{C}(R/ 2, R)} \vert \U \vert^3 \right)^{2/3}
 \leq c \,R^{\frac{2}{3}\gamma-1} \Vert \U \Vert^2_{M^{3}_{\gamma}}.$$ \\
and also through a same computation, we have
$$ R^{-1}\|\mathbf{u}\|^2_{L^{3}(\mathcal{C}(R/2,R))}
\leq c \,R^{\frac{2}{3}\gamma-1} \Vert \U
\Vert^2_{M^{3}_{\gamma}}.$$ Due to $\frac{2}{3}\gamma-1<0$, hence
(\ref{lim-1}) follows. Also as $R\rightarrow\infty$, the remains
term in the estimate above vanish depending on the parameter
$\gamma$ (see \cite{Jarrin20-1} for a detailed proof). Hence, we
have the identities $\U=0$ and $\mathbf{b} =0$ and therefore,
Theorem \ref{Th1} is proven.

\section{Proof of Theorem \ref{Theo:Lorentz}}

Before proving Theorem \ref{Theo:Lorentz}, we see the following
Caccioppoli type estimate:
\begin{lem}\label{Prop-Base} Let $\mathcal{C}(R/ 2, R) = \{ x
\in \R^3: R/2 < \vert x \vert < R \}$. If the solution $(u,b)$
verifies $u \in L^{p}_{loc}(\R^3)$ and $\nabla u \in
L^{\frac{p}{2}}_{loc}(\R^3)$   with  $3 \leq p< +\infty$, then for
all $R>1$ we have
    \begin{equation*}\label{Ineq-base}
    \begin{split}
      &  \int_{B_{R/2}} \vert \nabla  \mathbf{u} \vert^2
dx+\int_{B_{R/2}} \vert \nabla  \mathbf{b} \vert^2 dx \lesssim
\left( \int_{\mathcal{C}(R/2,R)}\vert \nabla  \mathbf{u}
\vert^{\frac{p}{2}} dx \right)^{\frac{2}{p}} R^{2-\frac{9}{p}}
\left( \int_{\mathcal{C}(R/2,R)} \vert \mathbf{u} \vert^{p} dx
\right)^{\frac{1}{p}}\\
& +\left( \int_{\mathcal{C}(R/2,R)}\vert \nabla  \mathbf{b}
\vert^{\frac{p}{2}} dx \right)^{\frac{2}{p}} R^{2-\frac{9}{p}}
\left( \int_{\mathcal{C}(R/2,R)} \vert \mathbf{b} \vert^{p}
dx \right)^{\frac{1}{p}}  \\
&  + \Big[\left( \int_{\mathcal{C}(R/2,R)} \vert \mathbf{u}\otimes
\mathbf{u} \vert^{\frac{p}{2}}dx \right)^{\frac{2}{p}}+\left(
\int_{\mathcal{C}(R/2,R)} \vert \mathbf{b}\otimes \mathbf{b}
\vert^{\frac{p}{2}}dx \right)^{\frac{2}{p}}\Big] R^{2-\frac{9}{p}}
\left( \int_{\mathcal{C}(R/2,R)} \vert \mathbf{u}\vert^{p} dx
\right)^{\frac{1}{p}}.
      \end{split}
    \end{equation*}
\end{lem}
\begin{proof}
Note that
\[
-\sum_{i,j=1}^{3}\int_{\R^3}\partial_{jj} v_i\cdot v_i\varphi_R\,dx
=\sum_{i,j=1}^{3}\int_{\R^3}\partial_{j} v_i\cdot
\partial_{j}v_i\varphi_R\,dx+\int_{\R^3}\partial_{j} v_i\cdot
v_i\partial_{j}\varphi_R\,dx
\]
\begin{equation}\label{aaa}
=\int_{\R^3}|\nabla
\mathbf{v}|^2\varphi_R\,dx+\sum_{i,j=1}^{3}\int_{\R^3}\partial_{j}
v_i\cdot v_i\partial_{j}\varphi_R\,dx.
\end{equation} Using \eqref{aaa},
\begin{eqnarray} \label{eq03} \nonumber
 \int_{B_{R/2}} (\vert \nabla \mathbf{u} \vert^{2} &+&\vert \nabla \mathbf{b} \vert^{2})\,dx  =
   - \sum_{i,j=1}^{3}\int_{B_R} \Big[(\partial_j u_i) (\partial_j \varphi_R) u_i +(\partial_j b_i) (\partial_j \varphi_R) b_i \Big]dx  \\ \nonumber
 & &-2(\lambda+\nu)\int_{\mathcal{C}(R/2,R)} \operatorname{div} \mathbf{u} \mathbf{u}\cdot \nabla \varphi_R \, dx
 + \int_{B_R}(u \cdot \nabla)u\cdot  \varphi_R u  dx\\ \nonumber
  & &   -\int_{\mathcal{C}(R/2,R)}|\mathbf{u}|^2 \nabla \varphi_R\cdot\rho \mathbf{u} \, dx +\int_{\mathcal{C}(R/2,R)}|\mathbf{b}|^2 \nabla \varphi_R\cdot \mathbf{u} \, dx, \\
 & & - \int_{\mathcal{C}(R/2,R)}\nabla P\cdot \mathbf{u}\varphi_R\, dx:=\sum_{i=1}^6\mathcal{M}_i.
 \end{eqnarray}
We study now these four terms above. In term $I_1$ remark that we
have the function $\partial_i \varphi_R$, but since the test
function $\varphi_R$ verifies $\varphi_R(1)$ if $\vert x \vert
<\frac{R}{2}$ and $\varphi_R(x)=0$  if $\vert x \vert
> R$ then we have $supp\,(\nabla \varphi_R) \subset
\mathcal{C}(R/2,R)$, and thus we can write
$$ \mathcal{M}_1= -\sum_{i,j=1}^{3} \int_{\mathcal{C}(R/2,R)} \partial_j u_i (\partial_j \varphi_R) u_i dx
-\sum_{i,j=1}^{3} \int_{\mathcal{C}(R/2,R)} \partial_j b_i
(\partial_j \varphi_R) b_i dx.$$

Then, applying the Hold\"er inequality, we have
\begin{eqnarray}\label{eq06} \nonumber
\int_{\mathcal{C}(R/2,R)} \partial_j u_i (\partial_j \varphi_R) u_i
dx&\lesssim &  \left( \int_{\mathcal{C}(R/2,R)}\vert \nabla u
\vert^{\frac{p}{2}} dx \right)^{\frac{2}{p}}  \frac{1}{R} \left(
\int_{\mathcal{C}(R/2,R)} \vert u \vert^{q}
 dx\right)^{\frac{1}{q}}\\ \nonumber
 &\lesssim &  \left( \int_{\mathcal{C}(R/2,R)}\vert \nabla u
\vert^{\frac{p}{2}} dx \right)^{\frac{2}{p}} R^{2-\frac{9}{p}}
\left(  \int_{\mathcal{C}(R/2,R)} \vert u \vert^{p} dx
\right)^{\frac{1}{p}}.
\end{eqnarray}
With this estimate at hand  $\mathcal{M}_1$ is bounded by
\begin{equation*}\label{eq07}
\left( \int_{\mathcal{C}(R/2,R)}\vert \nabla \mathbf{u}
\vert^{\frac{p}{2}} dx \right)^{\frac{2}{p}} R^{2-\frac{9}{p}}
\left(  \int_{\mathcal{C}(R/2,R)} \vert \mathbf{u} \vert^{p} dx
\right)^{\frac{1}{p}}+\left( \int_{\mathcal{C}(R/2,R)}\vert \nabla
\mathbf{b} \vert^{\frac{p}{2}} dx \right)^{\frac{2}{p}}
R^{2-\frac{9}{p}} \left( \int_{\mathcal{C}(R/2,R)} \vert \mathbf{b}
\vert^{p} dx \right)^{\frac{1}{p}}
\end{equation*}
In a same manner as $\mathcal{M}_1$, $\mathcal{M}_2$ is estimated by
\begin{equation*}\label{eq07}
\left( \int_{\mathcal{C}(R/2,R)}\vert \nabla \mathbf{u}
\vert^{\frac{p}{2}} dx \right)^{\frac{2}{p}} R^{2-\frac{9}{p}}
\left(  \int_{\mathcal{C}(R/2,R)} \vert \mathbf{u} \vert^{p} dx
\right)^{\frac{1}{p}}
\end{equation*}
We study the term $\mathcal{M}_4$ and $\mathcal{M}_5$ in
(\ref{eq03}).  we have
\begin{equation*}\label{eq11}
\mathcal{M}_3+\mathcal{M}_4+\mathcal{M}_5  \lesssim  \Big[\left(
\int_{\mathcal{C}(R/2,R)} \vert \mathbf{u}\otimes \mathbf{u}
\vert^{\frac{p}{2}}dx \right)^{\frac{2}{p}}+\left(
\int_{\mathcal{C}(R/2,R)} \vert \mathbf{b}\otimes \mathbf{b}
\vert^{\frac{p}{2}}dx \right)^{\frac{2}{p}}\Big] R^{2-\frac{9}{p}}
\left( \int_{\mathcal{C}(R/2,R)} \vert \mathbf{u}\vert^{p} dx
\right)^{\frac{1}{p}}.
\end{equation*}
Here, for the estimate $\mathcal{M}_3$, in addition, we use
$\|\rho\|_{L^\infty}<\infty$. Finally, From \eqref{pressure-decom}
and \eqref{pressure-relation}, we recall that
\[
\mathcal{M}_6=\frac{2a\gamma}{\gamma-1}\int_{\mathbb{R}^3}P_1
\rho\mathbf{u}\cdot(\varphi_R\nabla \varphi_R)\, dx,
\]
and from \eqref{3-6a}.
\[
|P_1\rho|\leq C(a, \|\rho\|_{L^\infty})(|p_1|+|p_2|).
\]
in a similar way,
\begin{equation*}\label{eq11}
\mathcal{M}_6 \lesssim  \Big[\left( \int_{\mathcal{C}(R/2,R)} \vert
\nabla \mathbf{u} \vert^{\frac{p}{2}}dx \right)^{\frac{2}{p}}+\left(
\int_{\mathcal{C}(R/2,R)} \vert |\mathbf{u}|^2 \vert^{\frac{p}{2}}dx
\right)^{\frac{2}{p}}\Big] R^{2-\frac{9}{p}} \left(
\int_{\mathcal{C}(R/2,R)} \vert \mathbf{u}\vert^{p} dx
\right)^{\frac{1}{p}}.
\end{equation*}
Collecting all $\mathcal{M}_i$, we proved Lemma \eqref{Prop-Base}.
\end{proof}

\begin{proof}[Proof of Theorem \ref{Theo:Lorentz}]
Following the arguments in \cite{Jarrin20}, for $1<p<r \leq q <
+\infty$ and for $R>1$ we have
\begin{equation}\label{estim-Lorentz}
\int_{B_R} \vert \U \vert^{p} dx \leq c \, \, R^{3(1-\frac{p}{r})}
\Vert \U \Vert^{p}_{L^{r,\infty}} \leq c   \, R^{3(1-\frac{p}{r})}
\Vert \U \Vert^{p}_{L^{r,q}}.
\end{equation}
On the other hand, since
$$ \U = - \frac{1}{\Delta} \left( \P \left(( \U \cdot \vec{\nabla}) \U \right)-\P \left(( \B \cdot \vec{\nabla}) \B \right) \right),$$
 we have for $i=1,2,3$
\[
 \partial_i \U = -  \sum_{j=1}^{3}  \frac{1}{\Delta} \left( \P \left( \partial_i  \partial_j (u_j \U) \right) -\P \left( \partial_i  \partial_j (B_j \B) \right) \right)
 = \sum_{j=1}^{3} \P \left( \mathcal{R}_{i} \mathcal{R}_{j} (u_j \U) -\mathcal{R}_{i} \mathcal{R}_{j} (B_j \B)  \right),
\]  where $\mathcal{R}_{i}= \frac{\partial_i}{\sqrt{-\Delta}}$ denotes the i-th Riesz transform. Thus, 
\begin{equation}\label{estim-Lorentz-2}
\int_{B_R} \vert \vec{\nabla} \U \vert^{\frac{p}{2}} dx \leq c\,
R^{3(1-\frac{p/2}{r/2})} \Big(\Vert \U \otimes  \U
\Vert^{\frac{p}{2}}_{L^{\frac{r}{2},\frac{q}{2}}}+\Vert \B \otimes
\B \Vert^{\frac{p}{2}}_{L^{\frac{r}{2},\frac{q}{2}}}\Big),\quad
3\leq p <+\infty.
\end{equation}
Similarly, we can know
\begin{equation}\label{estim-Lorentz-3}
\int_{B_R} \vert \vec{\nabla} \B \vert^{\frac{p}{2}} dx \leq c\,
R^{3(1-\frac{p/2}{r/2})} \Big(\Vert \U \otimes  \B
\Vert^{\frac{p}{2}}_{L^{\frac{r}{2},\frac{q}{2}}}+\Vert \B \otimes
\U \Vert^{\frac{p}{2}}_{L^{\frac{r}{2},\frac{q}{2}}}\Big),\quad
3\leq p <+\infty.
\end{equation}

Through Lemma \ref{Prop-Base},  we write for all $R>1$
\begin{equation} \label{eq13}
\int_{B_{R/2}} \vert \nabla \U \vert^2 dx+\int_{B_{R/2}} \vert
\nabla \B \vert^2 dx\lesssim G_{u}P_u+G_{b}P_b+(S_{u}+S_{b})P_u,
\end{equation}
where
\[
G_{u}:=c \left( R^{\frac{6}{r}} \left( \frac{1}{R^3}
\int_{\mathcal{C}(R/2,R)} \vert \nabla\otimes
\U\vert^{\frac{p}{2}}dx \right)^{\frac{2}{p}},\quad
S_{u}:=R^{\frac{6}{r}} \left(\frac{1}{R^3} \int_{\mathcal{C}(R/2,R)}
\vert \U \otimes \U\vert^{\frac{p}{2}}dx \right)^{\frac{2}{p}}
\right),
\]
\[
G_{b}:=c R^{\frac{6}{r}} \left( \frac{1}{R^3}
\int_{\mathcal{C}(R/2,R)} \vert \nabla\otimes
\B\vert^{\frac{p}{2}}dx \right)^{\frac{2}{p}},\quad
S_{b}:=R^{\frac{6}{r}} \left(\frac{1}{R^3} \int_{\mathcal{C}(R/2,R)}
\vert \B \otimes \B\vert^{\frac{p}{2}}dx \right)^{\frac{2}{p}},
\]
\[
P_u:=R^{2-\frac{9}{r}}  \left( R^{\frac{3}{r}} \left( \frac{1}{R^3}
\int_{\mathcal{C}(R/2,R)} \vert \U \vert^{p} dx
\right)^{\frac{1}{p}}\right),\quad P_b:=R^{2-\frac{9}{r}}  \left(
R^{\frac{3}{r}} \left( \frac{1}{R^3} \int_{\mathcal{C}(R/2,R)} \vert
\B \vert^{p} dx \right)^{\frac{1}{p}}\right).
\]
For  this  we introduce  the cut-off function  $\theta_R \in
\mathcal{C}^{\infty}_{0}(\R^3)$ such that  $\theta_R =1$ on  $
\mathcal{C}(R/2,R)$, $supp\, (\theta_R) \subset \mathcal{C}(R/4,
2R)$ and $\Vert \nabla \theta_R \Vert_{L^{\infty}} \leq
\frac{c}{R}$. In the same approach in \cite{Jarrin20} with
\eqref{estim-Lorentz}--\eqref{estim-Lorentz-3}, \eqref{eq13} yields
\begin{equation*}\label{eq18}
\int_{B_{\frac{B}{2}}} \Big(\vert \nabla \U \vert^2 dx+\vert \nabla
\B \vert^2 dx\Big) \leq c \left( \Vert \theta_R (\nabla
\U)\Vert_{L^{\frac{r}{2},\frac{q}{2}}} + \Vert \theta_R \U
\Vert^{2}_{L^{r,q}}\right) \, R^{2-\frac{9}{r}} \Vert \U
\Vert_{L^{r,q}(\mathcal{C}(R/4,R))},
\end{equation*}
\[
+ c \Vert \theta_R (\nabla
\B)\Vert_{L^{\frac{r}{2},\frac{q}{2}}}R^{2-\frac{9}{r}}    \Vert \B
\Vert_{L^{r,q}(\mathcal{C}(R/4,R))} + c \Vert \theta_R \B
\Vert^{2}_{L^{r,q}} R^{2-\frac{9}{r}} \Vert \U
\Vert_{L^{r,q}(\mathcal{C}(R/4,R))}.
\]
This implies $\U\equiv0\equiv\B$ in $\R^3$ under the assumptions in
Theorem \ref{Theo:Lorentz}.
\end{proof}

\section{Proof of Theorem \ref{mixed-norm}}
For each $q \in [1, \infty)$, a non-negative measurable function
$\omega: \mathbb{R}^n \rightarrow \mathbb{R}$ is said to be in the
Muckenhoupt  $A_q(\bR^n)$-class if
\[
[\omega]_{A_q} := \displaystyle{\sup_{R>0 , x_0 \in \bR^n}
\left(\frac{1}{|B_R(x_0)|} \int_{B_R(x_0)} \omega(x) dx \right)
\left(\frac{1}{B_R(x_0)} \int_{B_R(x_0)} \omega(x)^{-\frac{1}{q-1}}
dx  \right)^{q-1}} < \infty
\]
for  $q \in (1, \infty)$, and
\[
[\omega]_{A_1} := \displaystyle{\sup_{R>0 , x_0 \in \bR^n}
\left(\frac{1}{|B_R(x_0)|} \int_{B_R(x_0)} \omega(x) dx \right)
\|\frac{1}{\omega}\|_{L_\infty(B_R(x_0)}} < \infty,
\]
where $B_R(x_0)$ denotes the ball in $\bR^3$ with the radius $R$ and
centered at $x_0 \in \bR^3$.

We introduce the following lemma on weighted mixed norm estimates
for the pressure of the equations \eqref{MHD}. This lemma is an key
ingredient in the paper.
\begin{lem} \label{pressure-lemma} Let $q, r \in (2, \infty)$ and $M_0 \geq 1$. Assume that $\U, \B \in L_{q, r}(\R^3, \omega)$ with $\omega(x) =\omega_1(x') \omega_2(x_3)$ for all a.e. $x = (x', x_3) \in \R^2 \times \R$, and for $\omega_1 \in A_{\frac{q}{2}}(\R^2)$, $\omega_2 \in A_{\frac{r}{2}}(\R)$ with
\[
[\omega_1]_{A_{\frac{q}{2}}(\R^2)} \leq M_0, \quad
[\omega_2]_{A_{\frac{r}{2}}(\R)} \leq M_0.
\]
Then, there exists $N = N(q,r, M_0) >0$ such that
\[
\|P\|_{L_{\frac{q}{2}, \frac{r}{2}}(\R^3, \omega)} \leq N
\Big(\|\U\|_{L_{q, r}(\R^3, \omega)}^2+\|\B\|_{L_{q, r}(\R^3,
\omega)}^2\Big),
\]
where $P$ is defined as
\begin{equation} \label{pressure-formula}
P =\sum_{i, j =1}^3 \Big(\mathcal{R}_i \mathcal{R}_j (u_i
u_j)+\mathcal{R}_i \mathcal{R}_j (b_i b_j)\Big),
\end{equation}
in which $\mathcal{R}_i $ denotes the $i$-th Riesz transform.
\end{lem}

\begin{proof}
This proof is almost similar to that in \cite[Theorem 1.4]{Phan20}
and so we omit the proof.
\end{proof}

\begin{proof}[Proof of Theorem \ref{mixed-norm}]
Let $\phi \in C^\infty_0(\R)$ be a standard cut-off function with $0
\leq \phi \leq 1$ and
\[
\phi =1 \quad \text{on} \quad [-1/2, 1/2] \quad \text{and} \quad
\phi =0 \quad \text{on} \quad \R \setminus [-1, 1].
\]
For each $R>0$,  let $Q_R = [-R, R]^3$ and
$$\phi_R(x) =
\phi(\frac{x_1}{R}) \phi(\frac{x_2}{R})\phi(\frac{x_3}{R}), \quad
x=(x_1, x_2, x_3) \in \R^3.$$ Then, we see that $ \phi_R(x) = 1
\quad \text{on} \quad Q_{R/2}, \quad \phi =0 \quad \text{on} \quad
\R^3\setminus Q_R$. Then, there is a constant $C>0$ independent on
$R$ such that $ |\nabla^k \phi_R| \leq \frac{C}{R^k},\quad k=0,1,2$.
Testing $\U\phi_R$ and  $b\phi_R$ as a test function for the system
\eqref{MHD}, respectively, we have
\begin{align}  \nonumber
\int_{Q_{R/2}}& (|\nabla \U|^2+|\nabla \B|^2) dx\nonumber\\
&\leq \frac{1}{2}\int_{Q_R\setminus Q_{R/2}} (|\U|^2+|\B|^2) |\Delta
\phi_R| dx+ \frac{1}{2}\int_{Q_R\setminus Q_{R/2}} (|\U|^3+|\B|^3)
|\nabla \phi_R| dx\nonumber\\&+ \int_{Q_R\setminus Q_{R/2}} |p| |\U|
|\nabla \phi_R| dx\label{app_test},
\end{align}
where we use
\[
 \int_{Q_R\setminus Q_{R/2}}
|\U||\B|^2  |\nabla \phi_R| dx\leq \frac{1}{2} \int_{Q_R\setminus
Q_{R/2}} (|\U|^3+|\B|^3)  |\nabla \phi_R| dx
\]
From the arguments in  \cite{Phan20}, we can show that the three
terms in \eqref{app_test} becomes to $0$ as $R\rightarrow\infty$
under the assumptions and Lemma \ref{pressure-formula}. Indeed, the
proof of this part is almost same that in \cite{Phan20} and thus we
skip a proof.  Hence, we obtain
\[
\int_{\R^3} \Big(|\nabla  \U|^2+|\nabla  \B|^2\Big) dx =
\lim_{R\rightarrow \infty} \int_{Q_{R/2}}\Big(|\nabla \U|^2+|\nabla
\B|^2\Big) dx =0.
\]
Therefore, $(\U,\B)$ is a constant function in $\R^3$. From this and
the fact that $(\U,\B) \in L_{q,r}(\R^3)$, we conclude that $\U
\equiv 0\equiv\B$ in $\R^3$. The proof is then completed.
\end{proof}

\section*{Acknowledgements}
Jae-Myoung Kim was supported by a Research Grant of Andong National
University.

\end{document}